\documentclass[12pt]{article}
\usepackage[a4paper,margin=1.1in]{geometry}
\usepackage{amssymb}
\usepackage{amsmath}
\usepackage{amsfonts}
\usepackage{amsthm}
\usepackage{color}
\usepackage{float}
\usepackage[all]{xy}
\usepackage{scalerel}
\usepackage{mathabx}
\usepackage{delimset}
\usepackage{setspace}

\def\Exp{\mathop{\mbox{\textup{Exp}}}\nolimits}

\newcommand{\fibonomial}{\genfrac{\{}{\}}{0pt}{}}

\newcommand{\C}{\mathbb{C}}
\newcommand{\Z}{\mathbb{Z}}

\newcommand{\N}{\mathbb{N}}
\newcommand{\Q}{\mathbb{Q}}
\newcommand{\R}{\mathbb{R}}
\newcommand{\e}{\mathrm{e}}

\newtheorem{theorem}{Theorem}
\newtheorem{lemma}{Lemma}
\newtheorem{definition}{Definition}
\newtheorem{proposition}{Proposition}
\newtheorem{corollary}{Corollary}

\begin{document}

\title{\textbf{Irrationality of the Deformed Euler Numbers $\e_{s,t,u}$}}
\author{Ronald Orozco L\'opez}
\newcommand{\Addresses}{{
  \bigskip
  \footnotesize

  \textit{E-mail address}, R.~Orozco: \texttt{rj.orozco@uniandes.edu.co}
  
}}

\maketitle
\tableofcontents

\begin{abstract}
In this paper, we define the deformed Euler $(s,t)$-numbers $\e_{s,t,u}$ Furthermore, we prove that $\e_{as,a^2t,u^{-1}}$ and $\e_{as,a^2t,u^{-1}}^{-1}$ are irrational numbers when $a,u\in\Q$ and $\vert au\vert>1$, thus providing a countable infinite family of irrational numbers. This is the first step in a program to study the irrationality of $(s,t)$-analog of known numbers.
\end{abstract}


\section{Introduction}

One of the most important constants in mathematics is the Euler number $e$. It is well known that $e$ is irrational and many proofs of this fact have been given \cite{euler1, hermite, hurwitz, liou, mac, penesi, sondow}. 
In this paper, the following deformed analogue of the number $e$ defined on generalized Fibonacci numbers is given
\begin{equation*}
    \e_{s,t,u}=\sum_{n=0}^{\infty}\frac{u^{\binom{n}{2}}}{\brk[c]{n}_{s,t}!}.
\end{equation*}
When $s=2$ and $t=-1$, we obtain
\begin{equation}
    \e_{u}=\sum_{n=0}^{\infty}\frac{u^{\binom{n}{2}}}{n!}
\end{equation}
which is the deformed natural basis obtained from the deformed exponential function \cite{patade2}
\begin{equation*}
    \Exp(x,u)=\sum_{n=0}^{\infty}u^{\binom{n}{2}}\frac{x^n}{n!}.
\end{equation*}
When $s=1$ and $t=1$, we obtain the Fibonacci natural base \cite{pashaev_1}
\begin{equation}
    \e_{F}=\sum_{n=0}^{\infty}\frac{1}{F_{n}!},
\end{equation}
where $F_{n}$ are the Fibonacci numbers. The numbers $\e_{u}$ and $\e_{F}$ are not yet proven to be irrational. Therefore, the aim of this paper is to prove the irrationality of the more general numbers $\e_{s,t,u}$. To achieve all the above results, the parameters $s,t$ will be required to satisfy the condition $s^2+4t\geq0$. 

The paper is divided as follows. In Section 2 we deal with generalized Fibonacci numbers $\brk[c]{n}_{s,t}$, where $s,t\in\R$-$\{0\}$ and $s^2+4t\geq0$. The reason for stating the above condition is because under it the sequence of numbers $\vert\brk[c]{n}_{s,t}\vert$ is increasing, which is the natural property of the sequence $(0,1,2,3,4,5,\ldots)$. Next, we define the deformed $(s,t)$-exponential functions. In Section 3 we introduce the deformed Euler $(s,t)$-numbers, and an estimate of them is given. In Section 4 it is proven that $\e_{as,a^2t,u^{-1}}$ and $\e_{as,a^2t,u^{-1}}^{-1}$ are irrational numbers when $a,u\in\Q$ and $\vert au\vert>1$. 

\section{Preliminaries}

\subsection{Generalized Fibonacci numbers for $s^2+4t\geq0$}

The generalized Fibonacci numbers on the parameters $s,t$ are defined by
\begin{equation}\label{eqn_def_fibo}
    \brk[c]{n+2}_{s,t}=s\{n+1\}_{s,t}+t\{n\}_{s,t}
\end{equation}
with initial values $\brk[c]{0}_{s,t}=0$ and $\brk[c]{1}_{s,t}=1$, where $s\neq0$ and $t\neq0$. In \cite{sagan} this sequence is called the generalized Lucas sequence. A Lucas sequence $L_{n}$ is defined as $L_{n+2}=L_{n+1}+L_{n}$ for $n\geq2$, with the initial conditions $L_{0}=2$ and $L_{1}=1$. Since the Lucas sequence is not a special case of the sequence in Eq.(\ref{eqn_def_fibo}), we will insist on calling the latter the generalized Fibonacci numbers and we will reserve the name of generalized Lucas numbers for that sequence that generalizes the sequence $L_{n}$.

The $(s,t)$-Fibonacci constant is the ratio toward which adjacent $(s,t)$-Fibonacci numbers tend. This is the only positive root of $x^{2}-sx-t=0$. We will let $\varphi_{s,t}$ denote this constant, where
\begin{equation*}
    \varphi_{s,t}=\frac{s+\sqrt{s^{2}+4t}}{2}
\end{equation*}
and its conjugate is
\begin{equation*}
    \varphi_{s,t}^{\prime}=s-\varphi_{s,t}=-\frac{t}{\varphi_{s,t}}=\frac{s-\sqrt{s^{2}+4t}}{2}.
\end{equation*}
Three cases arise from the discriminant $\Delta=s^2+4t$: $\Delta>0$, which produces an increasing sequence of the absolute values of its terms. $\Delta = 0$, which yields sequences of the form $n(\pm i\sqrt{t})^{n-1}$. Finally, when $\Delta<0$, then $s\neq0$ and $t<0$ and and we obtain sequences of the form
\begin{equation*}
    \brk[c]{n}_{\sqrt{a},-b/4}=\frac{2}{\sqrt{b-a}}\left(\frac{b}{4}\right)^n\sin(\theta n)
\end{equation*}
where $\theta=\arctan\left(\sqrt{\frac{b}{a}-1}\right)$ and $a=s^2$, $b=-4t$.

Below are some important specializations of generalized Fibonacci numbers when $\Delta>0$.
\begin{enumerate}
    \item When $s=1,t=1$, then $\brk[c]{n}_{1,1}=F_n$, the Fibonacci numbers.
    \item When $s=2,t=1$, then $\brk[c]{n}_{2,1}=P_n$, where $P_n$ are the Pell numbers
    \begin{equation*}
        P_n=(0,1,2,5,12,29,70,169,408\ldots).  
    \end{equation*}
    \item When $s=1,t=2$, then $\brk[c]{n}_{1,2}=J_n$, where $J_n$ are the Jacosbthal numbers
    \begin{equation*}
        J_n=(0,1,1,3,5,11,21,43,85,171,\ldots).
    \end{equation*}
    \item When $s=3,t=-2$, then $\brk[c]{n}_{3,-2}=M_n$, where $M_n=2^n-1$ are the Mersenne numbers
    \begin{equation*}
        M_n=(0,1,3,7,15,31,63,127,255,\ldots).
    \end{equation*}
\end{enumerate}
The Binet's $(s,t)$-identity is
\begin{equation*}
    \brk[c]{n}_{s,t}=
    \frac{\varphi_{s,t}^{n}-\varphi_{s,t}^{\prime n}}{\varphi_{s,t}-\varphi_{s,t}^{\prime}}
\end{equation*}
As $\varphi_{us,u^2t}=u\varphi_{s,t}$ and $\varphi_{us,u^2t}^{\prime}=u\varphi_{s,t}^{\prime}$, then follows that $\brk[c]{n}_{us,u^2t}=u^{n-1}\brk[c]{n}_{s,t}$. Then for a non-zero complex number $u$ we will say that $\brk[c]{n}_{us,u^t}$ is an $u$-\textit{deformation} of $\brk[c]{n}_{s,t}$. We define the \textit{alternating generalized Fibonacci numbers} as the $-1$-deformation of sequence $\brk[c]{n}_{s,t}$, thus,
\begin{equation*}
    \brk[c]{n}_{-s,t}=(-1)^{n-1}\brk[c]{n}_{s,t}.
\end{equation*}
For example, the alternating Fibonacci numbers are: $0,1,-1,2,-3,5-8,\ldots$ and the alternating Pell numbers are: $0,1,-2,5,-12,29,-70,169,-408,\ldots$. Another important $u$-deformation is a $\varphi_{s,t}$-deformation, so $\brk[c]{n}_{s,t}=\varphi_{s,t}^{n-1}[n]_{q}$ with $q=\varphi_{s,t}^\prime/\varphi_{s,t}$.

On the other hand, when $s^{2}+4t=0$, $t<0$, we obtain the degenerate case of the $(s,t)$-Fibonacci numbers. When $\varphi_{s,t}\rightarrow\varphi_{s,t}^{\prime}$, we obtain
\begin{align*}
\lim_{\varphi_{s,t}\rightarrow\varphi_{s,t}^{\prime}}\frac{\varphi_{s,t}^{n}-\varphi^{\prime n}_{a,b}}{\varphi_{s,t}-\varphi^{\prime}_{a,b}}=n\varphi_{s,t}^{\prime(n-1)}.
\end{align*}
Likewise, when $\varphi_{s,t}^{\prime}\rightarrow\varphi_{s,t}$, then $\brk[c]{n}_{s,t}\rightarrow n\varphi
_{s,t}^{n-1}$. Therefore, this implies that $s\rightarrow\pm2i\sqrt{t}$, $t<0$, and that $\varphi_{s,t}=\varphi_{s,t}^{\prime}=\pm i\sqrt{t}$. In this way we obtain the $(\pm2i\sqrt{t},t)$-\textit{Fibonacci function}
\begin{align}\label{eqn_fibo_dege}
\brk[c]{n}_{\pm2i\sqrt{t},t}=n(\pm i\sqrt{t})^{n-1}
\end{align}
for all $t\in\R$, $t<0$. When $t=-1$, then $\brk[c]{n}_{\pm2i\sqrt{t},t}=\brk[c]{n}_{\mp2,-1}=n(\mp1)^{n-1}$. On the other hand, in the $q$-calculus the degenerate case is obtained when $q\mapsto1$. In this situation, the $q$-numbers $\brk[s]{n}_q$ tend to the integers $n$. Then  $\frac{\varphi_{s,t}^{\prime}}{\varphi_{s,t}}\mapsto1$ implies that $\varphi_{s,t}\mapsto\sqrt{-t}$ and $\varphi_{s,t}^{\prime}\mapsto\sqrt{-t}$. Therefore, if $t=-1$, then
\begin{align*}
    \lim_{\varphi_{s,-1}\mapsto1}\frac{\varphi_{s,-1}^{n}-\varphi_{s,-1}^{\prime n}}{\varphi_{s,-1}-\varphi_{s,-1}^{\prime}}=n.
\end{align*}
On the other hand, the sequence $\{0,1,2,3,4,5,6,\ldots\}$, the basis of all classical calculus, is strictly increasing, and clearly, none of its elements is 0 except itself. This fact is important for defining the factorial of a number $n$. Then, we want to establish the range of values that the parameters $s$ and $t$ can take such that $\brk[c]{n}_{s,t}$ is an increasing sequence and such that $\brk[c]{n}_{s,t}\neq0$ for $n\neq0$. We begin our analysis with the following lemma.

\begin{lemma}\label{lem_abs_nst}
For $s\neq0$ and $t\neq0$ in $\R$,
\begin{equation*}
    \vert\brk[c]{n}_{s,t}\vert=\brk[c]{n}_{\vert s\vert,t}.
\end{equation*}
\end{lemma}
\begin{proof}
Follow easily if $s>0$ and $t>0$. If $s>0$ and $t<0$, then $0<\varphi_{s,t}^\prime<\varphi_{s,t}$ and $\vert\brk[c]{n}_{s,t}\vert=\brk[c]{n}_{s,t}$. If $s<0$ and $t\neq0$, then for $s=-u$, $u>0$
\begin{align*}
    \varphi_{s,t}&=\frac{-u+\sqrt{u^2+4t}}{2}=-\frac{u-\sqrt{u^2+4t}}{2}=-\varphi_{-s,t}^\prime,\\
    \varphi_{s,t}^{\prime}&=\frac{-u-\sqrt{u^2+4t}}{2}=-\frac{u+\sqrt{u^2+4t}}{2}=-\varphi_{-s,t}.
\end{align*}
In this way
\begin{align*}
\brk[c]{n}_{s,t}&=\frac{\varphi_{s,t}^{n}-\varphi_{s,t}^{\prime n}}{\varphi_{s,t}-\varphi_{s,t}^{\prime}}=\frac{(-\varphi_{-s,t}^{\prime})^n-(-\varphi_{-s,t})^n}{-\varphi_{-s,t}^\prime+\varphi_{-s,t}}=(-1)^{n-1}\brk[c]{n}_{-s,t}
\end{align*}
and
\begin{align*}
    \vert\brk[c]{n}_{s,t}\vert=\vert(-1)^{n+1}\brk[c]{n}_{-s,t}\vert=\vert\brk[c]{n}_{-s,t}\vert=\brk[c]{n}_{-s,t}=\brk[c]{n}_{\vert s\vert,t}
\end{align*}
and thus we obtain the first result. 
\end{proof}

The following Lemma exhibits conditions for $s$ and $t$ to achieve the aims of this paper.

\begin{lemma}\label{lemma_prop_nst}
Set $s,t\in\Z$ such that $s\neq0$ and $\Delta>0$ and suppose $\vert q\vert\neq1$. Then
\begin{enumerate}
    \item $\brk[c]{n}_{\vert s\vert,t}$ is strictly increasing.
    \item If $\vert s\vert+t>1$ with $\vert s\vert\geq1$, then $\brk[c]{n+1}_{\vert s\vert,t}>\brk[c]{n}_{\vert s\vert,t}+1$, for all $n\geq2$.
    \item If $\vert s\vert+t>1$ with $\vert s\vert\geq1$, then $n\leq\brk[c]{n}_{\vert s\vert,t}$.
\end{enumerate}
\end{lemma}
\begin{proof}
\begin{enumerate}
    \item Take $s,t\in\Z$. As $s\neq0$, $t>0$ and $s^2+4t>0$, then every $\brk[c]{n}_{\vert s\vert,t}\in\N$ for all $n\in\N$. Now, suppose $0<\vert q\vert<1$. Then
    \begin{align*}
        \lim_{n\rightarrow\infty}\brk[c]{n}_{s,t}&=\lim_{n\rightarrow\infty}\frac{\varphi_{\vert s\vert,t}^n-\varphi_{\vert s\vert,t}^{\prime n}}{\varphi_{\vert s\vert,t}-\varphi_{\vert s\vert,t}^\prime}=\lim_{n\rightarrow\infty}\varphi_{\vert s\vert,t}^{n-1}\frac{1-q^n}{1-q}=\infty
    \end{align*}
    and $\brk[c]{n}_{\vert s\vert,t}$ is increasing. From Eq.(\ref{eqn_def_fibo}), $\brk[c]{n+2}_{s,t}>\brk[c]{n+1}_{s,t}$ and $\brk[c]{n+2}_{s,t}>\brk[c]{n}_{s,t}$. Then $\brk[c]{n}_{s,t}$ is strictly increasing. It is also proven to $\vert q\vert>1$. On the other side, suppose that $t<0$. Therefore must be $s^2-4t>0$ and thus $\vert s\vert>2\sqrt{t}$. We will prove using induction on $n$ that $\brk[c]{n}_{\vert s\vert,-t}\in\N$. It is easy to notice that $\brk[c]{2}_{\vert s\vert,-t}>2t^{1/2}$, $\brk[c]{3}_{\vert s\vert,-t}>3t^{2/2}$, $\brk[c]{4}_{s,-t}>4t^{3/2}$. We can therefore assume that $\brk[c]{n}_{\vert s\vert,-t}>nt^{(n-1)/2}$ and $\brk[c]{n+1}_{\vert s\vert,-t}>(n+1)t^{n/2}$. Thus
    \begin{align*}
        \brk[c]{n+2}_{\vert s\vert,-t}&=\vert s\vert\brk[c]{n+1}_{\vert s\vert,-t}-t\brk[c]{n}_{\vert s\vert,-t}\\
        &>2t^{1/2}(n+1)t^{n/2}-tnt^{(n-1)/2}=(n+2)t^{(n+1)/2}
    \end{align*}
    and $\brk[c]{n}_{\vert s\vert,-t}\in\N$ for all $n\in\N$. And it immediately follows that $\brk[c]{n}_{\vert s\vert,-t}<\brk[c]{n+1}_{\vert s\vert,-t}$ for all $n\in\N$. 
    
    \item Take $\vert s\vert\geq2$ and $t>0$ such that $\vert s\vert+t>1$. We will show by induction that $\brk[c]{n+1}_{\vert s\vert,t}>\brk[c]{n}_{\vert s\vert,t}+1$ for all $n\geq2$. When $n=2$, $\brk[c]{2}_{\vert s\vert,t}=\vert s\vert\geq2$ and as it is assumed that $\vert s\vert+t\geq1$, then $\brk[c]{3}_{\vert s\vert,t}=s^2+t>(1-t)^2+t=1-t+t^2>\brk[c]{2}_{\vert s\vert,t}+1$. Suppose it is true that $\brk[c]{n}_{\vert s\vert,t}>\brk[c]{n-1}_{\vert s\vert,t}+1$ and $\brk[c]{n+1}_{\vert s\vert,t}>\brk[c]{n}_{\vert s\vert,t}+1$. Then     
    \begin{align*}
        \brk[c]{n+2}_{\vert s\vert,t}&=\vert s\vert\brk[c]{n+1}_{\vert s\vert,t}+t\brk[c]{n}_{\vert s\vert,t}\\
        &>\vert s\vert(\brk[c]{n}_{\vert s\vert,t}+1)+t(\brk[c]{n-1}_{\vert s\vert,t}+1)\\
        &=\vert s\vert\brk[c]{n}_{\vert s\vert,t}+t\brk[c]{n-1}_{\vert s\vert,t}+\vert s\vert+t\\
        &\geq\brk[c]{n+1}_{\vert s\vert,t}+1
    \end{align*}
    and the statement is true for all $n\geq2$. Now, take $\vert s\vert=1$ and $t\geq1$. Then $\brk[c]{5}_{1,t}=1+3t+t^2>(1+2t)+1=\brk[c]{4}_{1,t}+1$ and by induction on $n$ we obtain $\brk[c]{n+1}_{1,t}>\brk[c]{n}_{1,t}+1$.
    
    \item If $t\geq1$, then $\brk[c]{n}_{\vert s\vert,-t}>nt^{(n-1)/2}>n$, for all $n\in\N$ and $\vert s\vert>2\sqrt{t}$. Now suppose that $\vert s\vert+t\geq1$ with $\vert s\vert>2$. Then $\brk[c]{2}_{\vert s\vert,t}>2$ and $\brk[c]{3}_{\vert s\vert,t}>3$. Assume by induction hypothesis that $\brk[c]{n}_{\vert s\vert,t}>n$. Then, according to statement 2,
\begin{equation*}
    \brk[c]{n+1}_{\vert s\vert,t}>\brk[c]{n}_{\vert s\vert,t}+1>n+1.
\end{equation*}
Then the statement is true for all $n\geq2$. Now take $\vert s\vert=1$ and $t\geq1$. Then $\brk[c]{6}_{1,t}=1+4t+3t^2>6$. Then, using statement 2 and proving by induction we arrive at the truth that $\brk[c]{n}_{1,t}>n$ for all $n\geq6$.

\end{enumerate}
\end{proof}
Take $t<0$ in Eq.(\ref{eqn_fibo_dege}). Then $\vert\brk[c]{n}_{\pm2\sqrt{a},-a}\vert$, with $a=-t$, is strictly increasing. Suppose $\Delta<0$. If $\arctan\sqrt{\frac{b}{a}-1}\neq\frac{k\pi}{n}$, for all $n\geq1$, then $\brk[c]{n}_{\sqrt{a},-b/4}\neq0$, for $n\neq0$. However, it is not an increasing sequence. For example,
\begin{equation*}
    \brk[c]{n}_{1,-2}=(0,1,1,-1,-3,-1,5,7,-3,\ldots).
\end{equation*}
For $s,t\in\Z$ such that $\Delta>0$, the $(s,t)$-Fibonomial coefficients are define by
\begin{equation*}
    \fibonomial{n}{k}_{s,t}=\frac{\brk[c]{n}_{s,t}!}{\brk[c]{k}_{s,t}!\brk[c]{n-k}_{s,t}!},
\end{equation*}
where $\brk[c]{n}_{s,t}!=\brk[c]{1}_{s,t}\brk[c]{2}_{s,t}\cdots\brk[c]{n}_{s,t}$ is the $(s,t)$-factorial or generalized fibotorial.
For $\Delta=0$, we define the $(\pm2i\sqrt{t},t)$-factorial and the $(\pm2i\sqrt{t},t)$-Fibonomial functions as 
\begin{align*}
    \brk[c]{n}_{\pm2i\sqrt{t},t}!&=(\pm i\sqrt{t})^{\binom{n}{2}}n!
\end{align*}
and
\begin{align*}
    \fibonomial{n}{k}_{\pm2i\sqrt{t},t}&=(\pm i\sqrt{t})^{k(n-k)}\binom{n}{k},
\end{align*}
respectively. For $\Delta<0$ with $\theta=\arctan\sqrt{\frac{b}{a}-1}\neq\frac{k\pi}{n}$, $k\in\Z$, $n\geq1$, the $(s,t)$-factorial is
\begin{equation*}
    \brk[c]{n}_{\sqrt{a},-b/4}!=\left(\frac{2}{\sqrt{b-a}}\right)^n\left(\frac{b}{4}\right)^{\binom{n+1}{2}}\prod_{k=1}^{n}\sin(\theta k)
\end{equation*}
and the $(s,t)$-Fibonomials coefficients are
\begin{equation*}
    \left(\frac{b}{4}\right)^{k(n-k)}\frac{\prod_{h=k+1}^{n}\sin(\theta h)}{\prod_{h=1}^{n-k}\sin(\theta h)}.
\end{equation*}

\subsection{Deformed $(s,t)$-exponential function}

Set $s,t\in\R-\{0\}$. We will set $q=\varphi_{s,t}^{\prime}/\varphi_{s,t}$ in the remainder of the paper. This section defines the deformed $(s,t)$-exponential function.
\begin{definition}
Set $s\neq0$. For all $u\in\C$, we define the deformed $(s,t)$-exponential function as
\begin{equation*}
    \exp_{s,t}(z,u)=
    \begin{cases}
        \sum_{n=0}^{\infty}u^{\binom{n}{2}}\frac{z^{n}}{\brk[c]{n}_{s,t}!}&\text{ if }u\neq0;\\
        1+z&\text{ if }u=0.
    \end{cases}
\end{equation*}
Also, we define 
\begin{align*}
\exp_{s,t}(z)&=\exp_{s,t}(z,1),\\
\Exp_{s,t}(z)&=\exp_{s,t}(z,\varphi_{s,t}),\\    
\Exp^{\prime}_{s,t}(z)&=\exp_{s,t}(z,\varphi^{\prime}_{s,t}).
\end{align*}
\end{definition}

It is straightforward to prove the following theorem.
\begin{theorem}\label{theo_exp_conv}
Set $s\neq0,t\neq0$. The function $\exp_{s,t}(z,u)$ is
\begin{enumerate}
    \item[1.] an entire function if either $(q,u)\in E_{1}$ or $(q,u)\in E_{2}$, where
    \begin{align*}
        E_{1}&=\{(q,u):0<\vert q\vert<1, 0<u<\vert\varphi_{s,t}\vert\},\\
        E_{2}&=\{(q,u):\vert q\vert>1, 0<u<\vert\varphi_{s,t}^\prime\vert\},
    \end{align*}
    \item[2.] convergent in the disks 
    \begin{align*}
        \mathbb{D}_{1}&=\{z\in\C:\vert z\vert<\vert\varphi_{s,t}\vert/\sqrt{s^{2}+4t}\},\text{ when } u=\vert\varphi_{s,t}\vert\text{ and }0<\vert q\vert<1,\\
        \mathbb{D}_{2}&=\{z\in\C:\vert z\vert<\vert\varphi_{s,t}^\prime\vert/\sqrt{s^{2}+4t}\},\text{ when }u=\vert\varphi_{s,t}^\prime\vert\text{ and }\vert q\vert>1,
    \end{align*}
         \item[3.] convergent in $z=0$ when either $u>\vert\varphi_{s,t}\vert$ or $u>\vert\varphi_{s,t}^\prime\vert$.
\end{enumerate}
\end{theorem}
Taking $s\rightarrow\pm2i\sqrt{t}$, $t<0$, then the deformed $(s,t)$-exponential functions reduce to the following deformed $(\pm2i\sqrt{t},t)$-exponential functions:
\begin{align*}
\exp_{\pm2i\sqrt{t},t}(z,u)&=\sum_{n=0}^{\infty}(u/\pm i\sqrt{t})^{\binom{n}{2}}\frac{z^{n}}{n!},\\
\exp_{\pm2i\sqrt{t},t}(z)&=\sum_{n=0}^{\infty}(\pm i\sqrt{t})^{-\binom{n}{2}}\frac{z^{n}}{n!},\\
\exp_{\pm2i\sqrt{t},t}^{\prime}(z)&=\sum_{n=0}^{\infty}(\pm i\sqrt{t})^{\binom{n}{2}}\frac{z^{n}}{n!},\\
\Exp_{\pm2i\sqrt{t},t}(z)&=\Exp_{F_{\pm2i\sqrt{t},t}}^{\prime}(z)=e^{z}.
\end{align*}
When $t=-1$, then
\begin{align*}
\exp_{\mp2,-1}(z,u)&=\sum_{n=0}^{\infty}(\mp u)^{\binom{n}{2}}\frac{z^{n}}{n!},\\
\exp_{\mp2,-1}(z)&=\exp^{\prime}_{\mp2,-1}(z)=\sum_{n=0}^{\infty}(\mp)^{\binom{n}{2}}\frac{z^{n}}{n!},\\
\Exp_{2,-1}(z)&=\Exp_{2,-1}^{\prime}(z)=e^{z}.
\end{align*}
Thus
\begin{equation*}
    \exp_{2,-1}(x,u)=\sum_{n=0}^{\infty}u^{\binom{n}{2}}\frac{x^n}{n!}
\end{equation*}
and
\begin{equation*}
    \exp_{2,-1}(x)=\exp^{\prime}_{2,-1}(x)=\Exp_{2,-1}(x)=\Exp^{\prime}_{2,-1}(x)=e^{x}.
\end{equation*}

\begin{theorem}
For all real number $t<0$ the function $\exp_{\pm2i\sqrt{t},t}(x,u)$
\begin{enumerate}
    \item is entire if $\vert u\vert\leq\vert\sqrt{t}\vert$.
    \item Converge in $x=0$ when $\vert u\vert>\vert\sqrt{t}\vert$.
\end{enumerate}
\end{theorem}

\begin{definition}
Set $s,t\in\R$, $s\neq0$, $t\neq0$. If $s^2+4t\neq0$, define the $(s,t)$-derivative $\mathbf{D}_{s,t}$ of the function $f(x)$ as
\begin{equation}
(\mathbf{D}_{s,t}f)(x)=
\begin{cases}
\frac{f(\varphi_{s,t}x)-f(\varphi_{s,t}^{\prime}x)}{(\varphi_{s,t}-\varphi_{s,t}^{\prime})x},&\text{ if }x\neq0;\\
f^{\prime}(0),&\text{ if }x=0
\end{cases}
\end{equation}
provided $f(x)$ differentiable at $x=0$. If $s^2+4t=0$, $t<0$, define the $(\pm i\sqrt{t},t)$-derivative of the function $f(x)$ as
\begin{equation}
    (\mathbf{D}_{\pm i\sqrt{t},t}f)(x)=f^{\prime}(\pm i\sqrt{t}x).
\end{equation}
If $(\mathbf{D}_{s,t}f)(x)$ exist at $x=a$, then $f(x)$ is $(s,t)$-differentiable at $a$.
\end{definition}

\begin{proposition}
For all non-zero real numbers $s,t$, the deformed $(s,t)$-exponential function satisfies the equation
\begin{equation}\label{eqn_def_exp}
    \mathbf{D}_{s,t}\exp_{s,t}(x,u)=\exp_{s,t}(ux,u),\ \exp_{s,t}(0,u)=1.
\end{equation}
\end{proposition}
Then the function $\exp_{s,t}(x,u)$ is the $(s,t)$-analog of the deformed exponential function 
\begin{equation*}
    \Exp(x,q)=\sum_{n=0}^{\infty}q^{\binom{n}{2}}\frac{x^n}{n!},
\end{equation*}
which satisfies the functional differential equation
\begin{equation*}
    y^{\prime}(x)=y(qx),\ y(0)=1.
\end{equation*}
The function $\Exp(x,y)$ is a deformed exponential function since when $q\rightarrow 1$, then $\Exp(x,y)\rightarrow e^x$. It is closely related to the generating function for the Tutte polynomials of the complete graph $K_n$ in combinatorics \cite{gessel2}, the partition function of one-site lattice gas with fugacity $x$ and two-particle Boltzmann weight $q$ in statistical mechanics \cite{sokal2}, cell division \cite{brunt}, and brightness of the galaxy \cite{ambart}.

\section{Deformed Euler $(s,t)$-numbers}

\subsection{Definition}
We want that $\vert\varphi_{s,t}\vert>1$ and that $\vert\varphi_{s,t}^\prime\vert>1$, since in this way $\exp_{s,t}(1,1)$ makes sense. With this in mind, the following sets are defined
\begin{align*}
    E_{11}^{*}&=\{(q,u)\in E_{1}:(t,s)\in[-1,\infty)\times(1-t,\infty)\text{ 
 or }(t,s)\in(-\infty,-1)\times[2\sqrt{-t},\infty)\},\\
 E_{12}^{*}&=\{(q,u)\in E_{1}:(t,s)\in(-\infty,-1)\times(t-1,-2\sqrt{-t}]\},\\
    E_{21}^{*}&=\{(q,u)\in E_{2}:(t,s)\in(-\infty,-1)\times[2\sqrt{-t},1-t)\},\\
    E_{22}^{*}&=\{(q,u)\in E_{2}:(t,s)\in[-1,\infty)\times(-\infty,t-1)\text{ or }(t,s)\in(-\infty,-1)\times(-\infty,-2\sqrt{-t}]\}.
\end{align*}

\begin{definition}
Set $s\neq0$, $t\neq0$. If $(q,u)\in E_{11}^{*},E_{12}^{*},E_{21}^{*},E_{22}^{*}$, we define the $u$-deformed Euler $(s,t)$-numbers as
    \begin{equation}\label{eqn_ep}
        \e_{s,t,u}\equiv\exp_{s,t}(1,u)=\sum_{n=0}^{\infty}\frac{u^{\binom{n}{2}}}{\brk[c]{n}_{s,t}!}.
    \end{equation}
Also, we denote
\begin{equation*}
    \e_{s,t}\equiv\exp_{s,t}(1,1)
\end{equation*}
and
\begin{equation*}
    \e_{s,t,u}^{-1}\equiv\sum_{n=0}^{\infty}(-1)^{n}\frac{u^{\binom{n}{2}}}{\brk[c]{n}_{s,t}!}.
\end{equation*}
\end{definition}
Some important specializations are
\begin{enumerate}
    \item Deformed Euler $(2,-1)$-number or deformed Euler number
    \begin{equation*}
        \e_{u}\equiv\e_{2,-1,u}=\sum_{n=0}^{\infty}\frac{u^{\binom{n}{2}}}{n!}
    \end{equation*}
    \item Deformed Euler $(1,1)$-number or deformed Fibonacci-Euler number
    \begin{equation*}
        \e_{F,u}\equiv\e_{1,1,u}=\sum_{n=0}^{\infty}\frac{u^{\binom{n}{2}}}{F_{n}!}
    \end{equation*}
    \item Deformed Euler $(2,1)$-number or deformed Pell-Euler number
    \begin{equation*}
        \e_{P,u}\equiv\e_{2,1,u}=\sum_{n=0}^{\infty}\frac{u^{\binom{n}{2}}}{P_{n}!}
    \end{equation*}
    \item Deformed Euler $(1,2)$-number or deformed Jacobsthal-Euler number
    \begin{equation*}
        \e_{J,u}\equiv\e_{1,2,u}=\sum_{n=0}^{\infty}\frac{u^{\binom{n}{2}}}{J_{n}!}
    \end{equation*}
    \item Deformed Euler $(3,-2)$-number or deformed Mersenne-Euler number
    \begin{equation*}
        \e_{M,u}\equiv\e_{3,-2,u}=\sum_{n=0}^{\infty}\frac{u^{\binom{n}{2}}}{M_{n}!}
    \end{equation*}
\end{enumerate}

Likewise, we want $\exp_{s,t}(1,u)$ to make sense either when $u=\vert\varphi_{s,t}\vert$ or when $u=\vert\varphi_{s,t}^\prime\vert$. So we define the following sets
\begin{align*}
    D_{11}^{*}&=\{(t,s)\in(-\infty,0)\times(2\sqrt{-t},\infty)\},\\
    D_{12}^{*}&=\{(t,s)\in(-\infty,0)\times(-3\sqrt{-2t}/2,-2\sqrt{-t})\},\\
    D_{21}^{*}&=\{(t,s)\in(-\infty,0)\times(2\sqrt{-t},3\sqrt{-2t}/2)\},\\
    D_{22}^{*}&=\{(t,s)\in(-\infty,0)\times(-\infty,-2\sqrt{-t})\}.
\end{align*}

\begin{definition}
Set $s\neq0$, $t\neq0$. If $(t,s)\in D_{11}^{*},D_{12}^{*}$, we define the $\varphi$-deformed Euler $(s,t)$-numbers as
    \begin{equation*}
        \e_{s,t,\varphi}\equiv\exp_{s,t}(1,\varphi_{s,t})=\sum_{n=0}^{\infty}\frac{\varphi_{s,t}^{\binom{n}{2}}}{\brk[c]{n}_{s,t}!}.
    \end{equation*}
Also, we denote
\begin{equation*}
    \e_{s,t,\varphi_{s,t}}^{-1}\equiv\sum_{n=0}^{\infty}(-1)^{n}\frac{\varphi_{s,t}^{\binom{n}{2}}}{\brk[c]{n}_{s,t}!}.
\end{equation*}
Equally, if $(t,s)\in D_{21}^{*},D_{22}^{*}$, we define the $\varphi^\prime$-deformed Euler $(s,t)$-numbers as
    \begin{equation*}
        \e_{s,t,\varphi^\prime}\equiv\exp_{s,t}(1,\varphi_{s,t}^\prime)=\sum_{n=0}^{\infty}\frac{\varphi_{s,t}^{\prime\binom{n}{2}}}{\brk[c]{n}_{s,t}!}.
    \end{equation*}
Also, we denote
\begin{equation*}
    \e_{s,t,\varphi_{s,t}^\prime}^{-1}\equiv\sum_{n=0}^{\infty}(-1)^{n}\frac{\varphi_{s,t}^{\prime\binom{n}{2}}}{\brk[c]{n}_{s,t}!}.
\end{equation*}
\end{definition}

\subsection{Estimating the numbers $\e_{s,t}$}

\begin{theorem}
Suppose that $\varphi_{s,t}>1$. Then
\begin{equation*}
    2+\frac{1}{s}+h(s,t)<\e_{s,t}<2+\frac{1}{s}+h(s,t)+\frac{1}{\brk[c]{7}_{s,t}!}\frac{1}{\brk[c]{8}_{s,t}-1}
\end{equation*}
where $h(s,t)=\frac{1}{\brk[c]{3}_{s,t}!}+\frac{1}{\brk[c]{4}_{s,t}}+\frac{1}{\brk[c]{5}_{s,t}}+\frac{1}{\brk[c]{6}_{s,t}!}$.
\end{theorem}
\begin{proof}
To estimate $\e_{s,t}$ we have the sum
\begin{align*}
    \e_{s,t}=\sum_{n=0}^{\infty}\frac{1}{\brk[c]{n}_{s,t}!}&=\frac{1}{\brk[c]{0}_{s,t}!}+\frac{1}{\brk[c]{1}_{s,t}!}+\frac{1}{\brk[c]{2}_{s,t}!}+h(s,t)+\sum_{n=7}^{\infty}\frac{1}{\brk[c]{n}_{s,t}!}\\
    &=2+\frac{1}{s}+h(s,t)+\sum_{n=7}^{\infty}\frac{1}{\brk[c]{n}_{s,t}!}
\end{align*}
 This gives the lower bound
\begin{equation*}
    2+\frac{1}{s}+h(s,t)<\e_{s,t}.
\end{equation*}
To get the upper bound, we combine 
\begin{align*}
    \e_{s,t}&=2+\frac{1}{s}+h(s,t)+\frac{1}{\brk[c]{7}_{s,t}!}\left(1+\frac{1}{\brk[c]{8}_{s,t}}+\frac{1}{\brk[c]{8}_{s,t}\brk[c]{9}_{s,t}}+\cdots\right).
\end{align*}
Since,
\begin{equation*}
    \brk[c]{n}_{s,t}>\brk[c]{8}_{s,t}\text{ for all }n>8, \Longrightarrow\frac{1}{\brk[c]{n}_{s,t}}<\frac{1}{\brk[c]{8}_{s,t}},
\end{equation*}
then
\begin{align*}
    \e_{s,t}&<2+\frac{1}{s}+h(s,t)+\frac{1}{\brk[c]{8}_{s,t}!}\left(1+\frac{1}{\brk[c]{8}_{s,t}}+\frac{1}{\brk[c]{8}_{s,t}^2}+\frac{1}{\brk[c]{8}_{s,t}^3}+\cdots\right)\\
    &<2+\frac{1}{s}+h(s,t)+\frac{1}{\brk[c]{8}_{s,t}!}\frac{1}{1-\frac{1}{\brk[c]{8}_{s,t}}}\\
    &<2+\frac{1}{s}+h(s,t)+\frac{1}{\brk[c]{7}_{s,t}!}\frac{1}{\brk[c]{8}_{s,t}-1}.
\end{align*}    
The proof is reached.
\end{proof}
The following approaches follow
\begin{align*}
    3.70416<&\e_{F}<3.70418,\\
    2.6086247947<&\e_{P}<2.6086247948,\\
    3.406355917<&\e_{J}<3.406355918,\\
    2.3842310161<&\e_{M}<2.3842310162.
\end{align*}
The estimate for $\e_{F}$ was obtained in \cite{merve}.

\section{$\e_{s,t,u^{-1}}$ and $\e_{s,t,u^{-1}}^{-1}$, $u\in\Q$, $\vert u\vert>1$, are irrationals}

\begin{theorem}\label{theo_e_irra}
Set $s,t\in\Z$ such that $s\neq0$, $t\neq0$ and $s^2+4t>0$ and set $u\in\Q$, $u>1$. Then $\e_{s,t,u^{-1}}$ is irrational.
\end{theorem}
\begin{proof}
Set $s>0$. First, we prove the following inequality
\begin{equation*}
    \e_{s,t,u^{-1}}-s_{n}<\frac{1}{u^{\binom{n+1}{2}}\brk[c]{n}_{s,t}!\brk[c]{n}_{s,t}},
\end{equation*}
where $s_{n}$ is the partial sum $\sum_{k=0}^{n}(1/u^{\binom{n}{2}}\brk[c]{n}_{s,t}!)$. Set $u\in\N$ and denote $s_{n}$ the partial sum of $\e_{s,t,u^{-1}}$
\begin{equation*}
    s_{n}=\sum_{k=0}^{n}\frac{1}{u^{\binom{k}{2}}\brk[c]{k}_{a,b}!}.
\end{equation*}
Then from Lemma \ref{lemma_prop_nst}
\begin{align*}
    \e_{s,t,u^{-1}}-s_{n}&=\frac{1}{u^{\binom{n+1}{2}}\brk[c]{n+1}_{s,t}!}+\frac{1}{u^{\binom{n+2}{2}}\brk[c]{n+2}_{s,t}!}+\frac{1}{u^{\binom{n+3}{2}}\brk[c]{n+3}_{s,t}!}+\cdots\\
    &=\frac{1}{u^{\binom{n+1}{2}}\brk[c]{n+1}_{s,t}!}\left(1+\frac{1}{u^{n+1}\brk[c]{n+2}_{s,t}}+\frac{1}{u^{2n+3}\brk[c]{n+2}_{s,t}\brk[c]{n+3}_{s,t}}+\right)\\
    &<\frac{1}{u^{\binom{n+1}{2}}\brk[c]{n+1}_{s,t}}\left(1+\frac{1}{\brk[c]{n+1}_{s,t}}+\frac{1}{\brk[c]{n+1}_{s,t}^2}+\cdots\right)\\
    &=\frac{1}{u^{\binom{n+1}{2}}\brk[c]{n+1}_{s,t}!}\frac{\brk[c]{n+1}_{s,t}}{\brk[c]{n+1}_{s,t}-1}\\
    &<\frac{1}{u^{\binom{n+1}{2}}\brk[c]{n}_{s,t}!\brk[c]{n}_{s,t}}.
\end{align*}
Therefore
\begin{equation}\label{eqn_esti_e}
    0<\e_{s,t,u^{-1}}-s_{n}<\frac{1}{u^{\binom{n+1}{2}}\brk[c]{n}_{s,t}!\brk[c]{n}_{s,t}}.
\end{equation}
Suppose that $\e_{s,t,u^{-1}}$ is a rational number, that is, $\e_{s,t,u^{-1}}=p/q$, where $p,q$ are positive integers. From Eq.(\ref{eqn_esti_e}) and Lemma \ref{lemma_prop_nst} statement 3,
\begin{equation*}
    0<u^{\binom{q+1}{2}}\brk[c]{q}_{s,t}!q(\e_{s,t,u^{-1}}-s_{q})<\frac{q}{\brk[c]{q}_{s,t}}<1
\end{equation*}
Set $u=\frac{n}{m}$, $n,m\in\N$. If $u$ is an integer, then both $u^{\binom{q+1}{2}}\brk[c]{q}_{s,t}!q\e_{s,t,u^{-1}}$ and
\begin{equation*}
    u^{\binom{q+1}{2}}\brk[c]{q}_{s,t}!qs_{q}=u^{\binom{q+1}{2}}\brk[c]{q}_{s,t}!q\left(1+\frac{1}{u^{\binom{1}{2}}\brk[c]{1}_{a,b}!}+\frac{1}{u^{\binom{2}{2}}\brk[c]{2}_{a,b}!}+\cdots+\frac{1}{u^{\binom{q}{2}}\brk[c]{q}_{a,b}!}\right)
\end{equation*}
are integers too. If $n$ and $m$ are co-prime and $m^{\binom{q+1}{2}}$ divide $\brk[c]{q}_{s,t}!(\e_{s,t,u^{-1}}-s_{q})$, then $u^{\binom{q+1}{2}}\brk[c]{q}_{s,t}!q(\e_{s,t,u^{-1}}-s_{q})$ is an integer between 0 and 1. If $m^{\binom{q+1}{2}}$ does not divide $\brk[c]{q}_{s,t}!(\e_{s,t,u^{-1}}-s_{q})$, then $u^{\binom{q+1}{2}}\brk[c]{q}_{s,t}!(\e_{s,t,u^{-1}}-s_{q})$ is a fraction larger than 1. Therefore the assumption that $\e_{s,t,u^{-1}}$ is a rational is false and for that reason $\e_{s,t,u^{-1}}$ is irrational.
If $s<0$, then 
\begin{equation*}
    0<\vert\e_{-s,t,u^{-1}}-s_{n}\vert<\frac{1}{u^{\binom{n+1}{2}}\brk[c]{n}_{\vert s\vert,t}!\brk[c]{n}_{\vert s\vert,t}}.
\end{equation*}
Now use the above argument. The proof is reached.
\end{proof}

\begin{theorem}\label{theo_e_inv_irra}
Set $s,t\in\Z$ such that $s\neq0$, $t\neq0$ and $s^2+4t\geq0$ and set $u\in\Q$, $u>1$. Then $\e_{s,t,u^{-1}}^{-1}$ is irrational.
\end{theorem}
\begin{proof}
Set $s>0$. First, we prove the following inequality
\begin{equation}
    \e_{s,t,u^{-1}}^{-1}-s_{2n-1}<\frac{1}{u^{\binom{n}{2}}\brk[c]{2n}_{s,t}!},
\end{equation}
where $s_{2n-1}$ is the partial sum $\sum_{k=0}^{2n-1}((-1)^k/u^{\binom{k}{2}}\brk[c]{k}_{s,t}!)$.
From Theorem 8.16 of \cite{apostol}
\begin{align*}
    \e_{s,t,u^{-1}}^{-1}-s_{2n-1}&=\sum_{k=0}^{\infty}(-1)^{k}\frac{1}{u^{\binom{k}{2}}\brk[c]{k}_{s,t}!}-\sum_{k=0}^{2n-1}(-1)^{k}\frac{1}{u^{\binom{k}{2}}\brk[c]{k}_{s,t}!}<\frac{1}{u^{\binom{2n}{2}}\brk[c]{2n}_{s,t}!}.
\end{align*}    
Note that if $u$ is an integer, then $u^{\binom{2n-1}{2}}\brk[c]{2n-1}_{s,t}!s_{2n-1}$ is always an integer. Assume that $\e_{s,t,u^{-1}}^{-1}$ is rational, so $p/q$, where $p$ and $q$ are co-prime, and $q\neq0$. Choose $n$ such that $n\geq(q+1)/2$. Then $u^{\binom{2n-1}{2}}\brk[c]{2n-1}_{s,t}!q\e_{s,t,u^{-1}}^{-1}$ is an integer too. Therefore, $u^{\binom{2n-1}{2}}\brk[c]{2n-1}_{s,t}!q(\e_{s,t,u^{-1}}^{-1}-s_{2n-1})$ is an integer less than $q/\brk[c]{2n}_{s,t}$, which is not possible because of the Lemma \ref{lemma_prop_nst}. Suppose that $u$ is a fraction of the form $n/m$, with $n>m$. If $m^{\binom{2n-1}{2}}$ divide to $\brk[c]{2n-1}_{s,t}!q(\e_{s,t,u^{-1}}^{-1}-s_{2n-1})$, then
\begin{equation}\label{eqn_proof_e}
    u^{\binom{2n-1}{2}}\brk[c]{2n-1}_{s,t}!q(\e_{s,t,u^{-1}}^{-1},s_{2n-1})
\end{equation}
is an integer. If $m^{\binom{2n-1}{2}}$ does not divide $\brk[c]{2n-1}_{s,t}!q(\e_{s,t,u^{-1}}^{-1}-s_{2n-1})$, then Eq(\ref{eqn_proof_e}) is a fraction larger than 1. Therefore, the assumption that $\e_{s,t,u^{-1}}^{-1}$ is a rational is false and thus $\e_{s,t,u^{-1}}^{-1}$ is irrational. If $s<0$, then use
\begin{equation*}
    0<\vert\e_{s,t,u^{-1}}^{-1}-s_{2n-1}\vert<\frac{1}{u^{\binom{n}{2}}\brk[c]{2n}_{\vert s\vert,t}!}
\end{equation*}
and the previous argument. The proof is achieved.
\end{proof}

\begin{proposition}
For all $a\in\R$ it holds that $\brk[c]{n}_{as,a^2t}=a^{n-1}\brk[c]{n}_{s,t}$, which implies that $\brk[c]{n}_{as,a^2t}!=a^{\binom{n}{2}}\brk[c]{n}_{s,t}!$.
\end{proposition}
\begin{proof}
It is easy to notice that if $a>0$, then $\varphi_{as,a^2t}=a\varphi_{s,t}$ and $\varphi_{as,a^2t}^{\prime}=a\varphi_{s,t}^\prime$. Then
\begin{align*}
    \brk[c]{n}_{as,a^2t}&=\frac{(a\varphi_{s,t})^n-(a\varphi_{s,t}^\prime)^n}{a(\varphi_{s,t}-\varphi_{s,t}^\prime)}=a^{n-1}\brk[c]{n}_{s,t}.
\end{align*}
If $a<0$, then $\varphi_{as,a^2t}=a\varphi_{s,t}^\prime$ and $\varphi_{as,a^2t}^{\prime}=a\varphi_{s,t}$. Therefore, $\brk[c]{n}_{as,a^2t}=a^{n-1}\brk[c]{n}_{a,t}$ for all $a\in\R$, $a\neq0$.    
\end{proof}

\begin{corollary}
Set $s,t\in\Z$ such that $s>0$, $t\neq0$ and $s^2+4t\geq0$ and set $u\in\Q$, $u>1$. Then  $\e_{as,a^2t,u^{-1}}$ and $\e_{as,a^2t,u^{-1}}^{-1}$ are irrational, for all $a\in\Q$ such that $\vert au\vert>1$.
\end{corollary}
\begin{proof}
From Proposition \ref{lemma_prop_nst}
\begin{align*}
    \e_{as,a^2t,u^{-1}}&=\sum_{n=0}^{\infty}\frac{u^{-\binom{n}{2}}}{\brk[c]{n}_{as,a^2t}!}=\sum_{n=0}^{\infty}\frac{u^{-\binom{n}{2}}}{a^{\binom{n}{2}}\brk[c]{n}_{s,t}!}=\sum_{n=0}^{\infty}\frac{(au)^{-\binom{n}{2}}}{\brk[c]{n}_{s,t}}=\e_{s,t,(au)^{-1}}.
\end{align*}
It is also shown that $\e_{as,a^2t,u^{-1}}^{-1}=\e_{s,t,(au)^{-1}}^{-1}$.
Finally, the Theorems \ref{theo_e_irra} and \ref{theo_e_inv_irra} must be applied. 
\end{proof}

\end{document}